\documentclass[reqno,10pt]{amsart}

\usepackage{amsmath,amsfonts,amssymb,amsthm,epsfig}

\usepackage[pdftex]{color}
\usepackage{pdfsync}
\usepackage{ifpdf}
\ifpdf
\fi

\usepackage{mparhack}
\usepackage[normalem]{ulem}
\usepackage{cancel} 

\definecolor{LightYellow}{cmyk}{0,0,0.4,0}
\definecolor{DarkGreen}{rgb}{0,0.5,0}

\newtheorem{teo}{Theorem}[section]
\newtheorem{prop}[teo]{Proposition}

\newtheorem{lemma}[teo]{Lemma}

\newtheorem{oss}[teo]{Remark}

\newtheorem*{teoS}{Selection Principle}

\newcommand{\compact}{\subset\subset}    
\newcommand{\de}{\partial}

\newcommand{\eps}{\varepsilon}
\newcommand{\e}{\varepsilon}

\newcommand{\mdiv}{\mathop{\mathrm{div}}}

\newcommand{\bary}{{\mathop{\mathrm{bar}}}}

\newcommand{\graph}{\mathop{\mathrm{gr}}}
\newcommand{\sgr}{\mathop{\mathrm{sgr}}}

\newcommand{\esssup}{\mathop{\mathrm{ess~sup}}}
\newcommand{\difsim}{\bigtriangleup}

\newcommand{\devmin}{\Psi}

\DeclareMathOperator*{\excess}{Exc}
\newcommand{\carat}[1]{\chi_{#1}}

\newcommand{\Per}{{P}}

\newcommand{\R}{{\mathbb R}}

\newcommand{\N}{{\mathbb N}}

\newcommand{\Se}{{\mathcal S}}

\newcommand{\Hau}{{\mathcal H}}

\newcommand{\asym}{{\alpha}}
\newcommand{\defP}{{\delta P}}

\newcommand{\QM}{{\mathcal{QM}}}

\newcommand{\impl}{\mathop{\Rightarrow}}

\newcommand{\cyl}{{\mathcal C}}
\newcommand{\xnp}{x_\nu^\perp}

\newcommand{\xenp}{x_{e_n}^\perp}

\newcommand{\xn}{x_\nu}
\newcommand{\xen}{x_{e_n}}

\title[A Selection Principle for the Sharp Quantitative Isoperimetric Inequality]{A Selection Principle for the Sharp Quantitative \\Isoperimetric Inequality}

\author[M.~Cicalese]{Marco Cicalese}
\address{Dipartimento di Matematica e Applicazioni ``R. Caccioppoli'', Universit{\`a} degli Studi di Napoli ``Federico II'', Via Cintia, 
Monte S. Angelo, I-80126 Napoli, Italy}
\email{cicalese@unina.it}

\author[G.P.~Leonardi]{Gian Paolo Leonardi}
\address{Dipartimento di Matematica Pura e Applicata ``G. Vitali'', Universit{\`a} degli Studi di Modena e Reggio Emilia, Via Campi 213/b, I-41100
    Modena, Italy}
\email{gianpaolo.leonardi@unimore.it}


\begin{document}

\begin{abstract}
We introduce a new variational method for the study of stability in the isoperimetric inequality.
The method is quite general as it relies on a penalization technique combined with the regularity theory for
quasiminimizers of the perimeter. Two applications are presented. First we give a new
proof of the sharp 
quantitative isoperimetric inequality in $\R^n$. Second we positively answer to a conjecture by Hall concerning the best constant for the quantitative isoperimetric inequality in $\R^2$ in the small asymmetry regime.
\end{abstract}

\maketitle

\section{Introduction}

Let $E$ be a Borel set in $\R^n$, $n\geq 2$, with positive and finite Lebesgue measure 
$|E|$. Denoting by $B_E$ the open ball centered at $0$ such that $|B_E| = |E|$, and
by $\Per(E)$ the distributional perimeter of $E$ (in the sense of Caccioppoli-De Giorgi), we define 
the \textit{isoperimetric deficit} of $E$ as
\begin{equation*}
  \defP(E) = \frac{\Per(E) - \Per(B_E)}{\Per(B_E)}.
\end{equation*}
By the classical isoperimetric inequality in $\R^n$, $\defP(E)$ is non-negative and zero
if and only if $E$ coincides with $B_E$ up to null sets and
to a translation. A natural issue arising from the optimality of the ball in the isoperimetric inequality, is that of stability estimates of the type
\begin{equation*}
\defP(E)\geq\varphi(E),
\end{equation*}  
where $\varphi(E)$ is a measure of how far $E$ is from a ball.  Such inequalities, called {\it Bonnesen-type inequalities} by Osserman (\cite{Osserman79}), have been widely studied after the results by Bernstein (\cite{Bernstein05}) and Bonnesen (\cite{Bonnesen21, Bonnesen24}) in the convex, $2$-dimensional case  (see also \cite{Hadwiger48} and \cite{Dinghas48} for extensions to convex sets in higher dimensions). Among inequalities of this kind, the well-known {\it quantitative isoperimetric inequality} states that there exists a constant $C = C(n)>0$, such that 
\begin{equation}
  \label{QII}
  \defP(E) \geq C \asym(E)^2,
\end{equation}
where
\begin{equation*}
  \asym(E) = \inf\left\{ \frac{|E\difsim (x+B_E)|}{|B_E|},\ x\in
    \R^n\right\} 
\end{equation*}
and $V\difsim W = (V\setminus W)\cup (W\setminus V)$. We
recall that $\asym(E)$ is known as the {\it Fraenkel asymmetry} of $E$
(see \cite{HalHayWei91}). Observe that both $\defP(E)$ and
$\asym(E)$ are invariant under isometries and dilations. For this
reason, denoting by $B$ the unit open ball in $\R^n$,  in studying \eqref{QII}
we are allowed to restrict ourselves to sets $E$ with $|E|=|B|$.

Before the complete proof of the inequality \eqref{QII} by Fusco,
Maggi and Pratelli \cite{FusMagPra08}, a number of partial results
came one after the other. 
A first stability result outside the convex setting was proved by Fuglede in \cite{Fuglede86} (see also \cite{Fuglede89}), who gave a proof of \eqref{QII} in the class of 
\textit{nearly-spherical} sets in $\R^n$. A set $E$ is \textit{nearly-spherical} in the sense of Fuglede if 
$\de E$ can be represented as the normal graph of a
Lipschitz function $u$ defined on $\de B$ and such that
$\|u\|_{W^{1,\infty}(\de B)}$ is suitably small. More specifically, the following inequality between the
isoperimetric deficit $\defP(E)$ and the Sobolev norm of $u$ is proved in \cite{Fuglede86} under the
assumption that $E$ is nearly-spherical and has the same barycenter as $B$:
\begin{equation}\label{Fug_intro}
 \defP(E)\geq C\| u \|^2_{W^{1,2}(\de B)},
\end{equation}
where $C=C(n)>0$. By \eqref{Fug_intro} one easily obtains \eqref{QII} (see Section \ref{sect:app}). 

A few years later, Hall proved in \cite{Hall92} the inequality \eqref{QII} for sets
with an axis of rotational symmetry (axisymmetric sets). Combining this result 
with a previous estimate obtained in \cite{HalHayWei91}, he was able to prove the quantitative isoperimetric inequality for all
sets in $\R^n$, but with a sub-optimal exponent ($4$ instead of $2$) for the asymmetry.

The full proof of the quantitative isoperimetric inequality (with the
sharp exponent $2$, as conjectured by Hall in \cite{Hall92}) has been
recently accomplished 
by Fusco, Maggi and Pratelli in \cite{FusMagPra08}, via a ingenious geometric construction by which the proof of \eqref{QII} is reduced to sets having more and more symmetries and eventually to axisymmetric sets, for which Hall's result leads to the conclusion.

Since the publication of \cite{FusMagPra08} the study of quantitative forms of various geometric and functional
inequalities has received a new impulse (see for instance \cite{FusMagPra07}, \cite{FusMagPra09}, \cite{FigMagPra09}, \cite{FigMagPra09-2}, \cite{CiaFusMagPra10} and the review paper \cite{Maggi08}). Among the recent results on this subject, the one by
Figalli, Maggi and Pratelli \cite{FigMagPra10} is of particular interest since the authors develop a new technique to study the stability in isoperimetric inequalities. More precisely, they show a more general version of \eqref{QII}, namely a quantitative version of the Wulff theorem, and their analysis relies on Gromov's proof of the isoperimetric inequality \cite{MilSch86} and on the theory of optimal mass transportation.

In this paper we develop a technique that we call \textit{Selection Principle}, which allows us to drastically restrict the class of sets on which the stability for the isoperimetric inequality has to be proved. The Selection Principle basically combines a penalization technique with the regularity theory for
quasiminimizers of the perimeter. We point out that the main ideas of this method work in more general frameworks, but we present them here in a form which is tailored to study the stability for the isoperimetric inequality. Indeed, as a first application of our technique we present a new proof of the sharp quantitative isoperimetric inequality in $\R^n$. 

We start from the simple observation that \eqref{QII} is equivalent to
\begin{equation}\label{DIQQ}
\frac{\defP(E)}{\asym(E)^2}\geq C
\end{equation}
when $\asym(E)>0$ (i.e., when $E$ is not a ball up to null sets). On the other hand, since $\asym(E)<2$, it is enough to show \eqref{DIQQ} under a smallness assumption on $\defP(E)$. This, in turns, translates into a smallness assumption on $\asym(E)$ (see Section \ref{sect:selprin}). Therefore, we only need to estimate from below the left-hand side of \eqref{DIQQ} in the {\it small asymmetry regime}, that is, as $\asym(E)$ gets smaller and smaller. To study the quotient on the left hand side of \eqref{DIQQ} in this regime we introduce the functional $Q$ defined as
\begin{equation*}
  Q(E) = \inf \Big\{ \liminf_k \frac{\defP(F_k)}{\asym(F_k)^2}:\ |F_k|=|E|,\ \asym(F_k)>0,\ |F_k\difsim E|\to 0\Big\}.
\end{equation*}
By the definition of $Q$, the inequality \eqref{DIQQ} in the small asymmetry regime turns out to be equivalent to the inequality 
\begin{equation}\label{QB_intro}
Q(B)>0. 
\end{equation}

The Selection Principle, that we state below, allows us to compute $Q(B)$ as the limit of $Q(E_j)$, as $j\to \infty$, and where $(E_j)_j$ is an ``optimal'' sequence of sets with asymmetry going to zero. More precisely, we prove in Section \ref{sect:selprin} the following result:
\begin{teoS}
There exists a sequence of sets $(E_j)_j$ with the following properties:
\begin{itemize}
\item[(i)] $|E_j| = |B|$, $\asym(E_j)>0$ and $\asym(E_j) \to 0$ as $j\to \infty$;

\item[(ii)] $Q(E_j) \to Q(B)$ as $j\to \infty$;

\item[(iii)] there exists a function $u_j\in C^1(\de B)$ such
  that $\de E_j = \{(1+u_j(x))x:\ x\in \de B\}$ and $u_j \to 0$ in the
  $C^1$-norm, as $j\to \infty$;

\item[(iv)] $\de E_j$ has (scalar) mean curvature $H_j\in
  L^\infty(\de E_j)$ and $\|H_j - 1\|_{L^\infty(\de E_j)} \to 0$ as $j\to \infty$.
\end{itemize}
\end{teoS}
As a striking consequence of the Selection Principle, in Theorem \ref{teo:QII} we obtain a new and very short proof of the quantitative isoperimetric inequality in $\R^n$. Indeed, thanks to (iii) and for $j$ large enough, $E_j$ is a nearly spherical set, and thus, by Fuglede's estimate \eqref{Fug_intro}, we have that $Q(E_j)\geq C$ for some $C=C(n)>0$. Eventually passing to the limit in $j$, we get \eqref{QB_intro}. 

In addition, in Theorem \ref{teo:Hconj} we positively answer to another conjecture posed by Hall in \cite{Hall92}, and asserting that for any measurable set in $\R^2$ with positive and finite Lebesgue measure the following {\it Taylor-type lower bound} holds true:
\begin{equation}
  \label{eq:HHWconj}
 \defP(E) \geq C_0 \asym(E)^2 + o(\asym(E))^2,
\end{equation}
with \textit{optimal asymptotic constant} $C_0=\frac{\pi}{8(4-\pi)}$. The inequality \eqref{eq:HHWconj} was already established in 
\cite{HalHayWei91, HalHay93} for convex sets in the plane. By property (iv) of the Selection Principle and for $j$ large enough, it turns out that $E_j$ is a convex set, hence by the validity of \eqref{eq:HHWconj} for convex sets,
\begin{eqnarray*}
Q(E_j)\geq C_0+o(1).
\end{eqnarray*}
Passing to the limit as $j\to \infty$ we get $Q(B)\geq C_0$ which
immediately implies \eqref{eq:HHWconj} for all Borel sets in $\R^2$
with positive and finite Lebesgue measure. Actually, an even more
precise estimate than \eqref{eq:HHWconj} can be proved. Indeed, in the
forthcoming paper \cite{CicLeo10-2}, relying on a more refined version of the
Selection Principle, we show in a rather direct way how to compute any
order of the optimal Taylor-type lower bound of the isoperimetric
deficit in terms of powers of the asymmetry (this result extend to all
Borel sets a former one obtained in \cite{AlvFerNit10} for convex sets
in the plane).

We conclude this introduction by briefly describing the main ideas of the proof of the Selection Principle. First, we construct a suitable sequence of penalized
functionals $(Q_j)_j$ defined as 
\begin{equation*}
  Q_j(E) = Q(E) + \left(\frac{\asym(E)}{\asym(W_j)}-1\right)^2,
\end{equation*}
where $(W_j)_j$ is a recovery sequence for $Q(B)$. Then, in Lemma \ref{lemma:exist-approx} we check
that $Q_j$ admits a minimizer $E_j$ enjoying a number of useful
properties. First of all, the sequence $(E_j)_j$ is a recovery sequence
for $Q(B)$, that is (ii) in the statement of the Selection Principle. Moreover, we can show in Lemma \ref{lemma:Lambda-min} that each $E_j$ is a quasiminimizer
of the perimeter (more specifically, a strong $\Lambda$-minimizer, see
Section \ref{sect:selprin} and \cite{Ambrosio97}). Therefore, in Lemma \ref{lemma:regconv}, we can
appeal to the regularity theory for quasiminimizers of the perimeter
(see \cite{DeGiorgi61}, \cite{Massari74}, \cite{Tamanini82},
\cite{Tamanini84}, \cite{Almgren76}) to get the property (iii) stated in the Selection Principle. In addition, by a first variation argument, in Lemma
\ref{lemma:firstvar}, we obtain (iv).

\medskip


\section{Notation and preliminaries}
\label{sect:prelim}
Given a Borel set $E\subset \R^n$, we denote by $|E|$ its
Lebesgue measure. Let $x\in \R^n$ and
$r>0$ be given, then we denote $B(x,r)$  
 as the open ball in $\R^n$ centered at
$x$ and of radius $r$. We also set $B=B(0,1)$ and $\omega_n=|B|$.  
Given $E\in\R^n$ we also denote by $\chi_E$ its characteristic function and we say that a sequence of sets $E_j$ converges
to $E$ with respect to the $L^1$ or the $L^1_{\rm loc}$-convergence of sets if $\chi_{E_j}\to \chi_E$ in $L^1$ or in $L^1_{\rm loc}$, respectively.
We recall that the \textit{perimeter} of a Borel set $E$ inside 
an open set $\Omega\subset \R^n$ is 
\[
\Per(E,\Omega) := \sup \left\{\int_E \mdiv g(x)\, dx:\ g\in
  C^1_c(\Omega;\R^n),\ |g|\leq 1\right\}. 
\]
This definition extends the natural notion of $(n-1)$-dimensional area
of a smooth (or Lipschitz) boundary $\de E$. 
We will say that $E$ has \textit{finite perimeter} in $\Omega$
if $P(E,\Omega)<\infty$. Equivalently, $E$ is a set of finite
perimeter in $\Omega$ if the distributional derivative $D\carat{E}$ of its characteristic
function $\carat{E}$ is a vector-valued Radon measure in $\Omega$ with finite total variation $|D\carat{E}|(\Omega)$. We will simply write $\Per(E)$ instead of
$\Per(E,\R^n)$, and we will say that $E$ is a set of finite perimeter
if $\Per(E)<\infty$. 
From the well-known De Giorgi's Rectifiability Theorem (see \cite{AmbFusPal00}, \cite{EvaGar92}), 
$D\carat{E}=\nu_E\,\Hau^{n-1}\lfloor\de^*E$ where $\Hau^{n-1}$ is the
$(n-1)$-dimensional Hausdorff measure and $\de^*E$ is the {\it reduced boundary} of $E$, i.e., the set of points $x\in\de E$ such that the {\it generalized inner normal} $\nu_E(x)$ is defined, that is,
\begin{equation*}
\nu_E(x)=\lim\limits_{r\to 0}\frac{D\carat{E}(B(x,r))}{|D\carat{E}|(B(x,r))}\quad\text{and}\quad|\nu_E(x)|=1.
\end{equation*}
We recall (see for instance \cite{Giusti84}) that, given $E$ of finite perimeter in $\R^n$, for all $x\in\de^*E$
\begin{equation}\label{density}
\lim_{r\to 0^+}\frac{|E\cap B(x,r)|}{|B(x,r)|}=\frac12,
\end{equation}
and for a.e. $r\in\R$ it holds that
\begin{eqnarray}\label{giusti}
\Per(E, B(0,r))= \Per(E\cap B(0,r))- \Hau^{n-1}(E\cap \de B(0,r)).
\end{eqnarray}

We now recall some classical
definitions and properties of quasiminimizers of the perimeter (see \cite{Tamanini82}, \cite{Tamanini84}, \cite{Ambrosio97}). 
Given $E\subset \R^n$ of finite perimeter and $A\subset \R^n$ an open
bounded set, we define the
\textit{deviation from minimality} of $E$ in $A$ as 
\[
\devmin(E,A) = \Per(E,A) - \inf\{\Per(F,A):\ F\difsim E \subset\subset A\},
\]
where $F\difsim E = (F\setminus E) \cup (E\setminus F)$ and $S\subset\subset A$ iff $S$ is a relatively compact subset of
$A$. Note that $\devmin(E,A) \geq 0$, with equality if
and only if $E$ minimizes the perimeter in $A$ (w.r.t. all of its compact
variations $F$). We
set $\devmin(E,x,r) = \devmin(E,B(x,r))$ and, given $\gamma\in (0,1)$,
$R>0$ and $\Lambda>0$, we call \textit{quasiminimizer} of the
perimeter (in $\R^n$) any set $E$ of finite perimeter for which 
\begin{equation}\label{def:qmin}
\devmin(E,x,r) \leq \Lambda \omega_{n-1} r^{n-1+2\gamma}
\end{equation}
for all $x\in \R^n$ and $0<r<R$ (see \cite{Tamanini82,
  Ambrosio97}). We will also equivalently write $E\in
\QM(\gamma, R, \Lambda)$ to highlight the key parameters
occurring in the above definition of quasiminimality. 
If $E\in \QM(\frac 12,R,\Lambda)$, i.e. when $\gamma = \frac 12$ in
\eqref{def:qmin}, then we call $E$ a \textit{$\Lambda$-minimizer} (see \cite{Ambrosio97}). Finally, if $E$ satisfies
\[
\Per(E,B(x,r)) \leq \Per(F,B(x,r)) + \Lambda \omega_{n-1}
\frac{|E\difsim F|}{\omega_n} 
\]
for all $x\in \R^n$, $0<r<R$ and all Borel sets $F$ such that $E\difsim F\subset\subset B(x,r)$, then $E$ is said to be a \textit{strong}
$\Lambda$-minimizer. It is easy to check that any strong
$\Lambda$-minimizer is also a $\Lambda$-minimizer, hence a
quasiminimizer of the perimeter (we refer to \cite{Tamanini84} for a clear treatment of the subject).

We now extend the definition of quasiminimality to sequences of sets
of finite perimeter. We say that a sequence $(E_h)_h$ of sets of
finite perimeter is a 
\textit{uniform sequence of quasiminimizers} if
$E_h\in \QM(\gamma,R,\Lambda)$ for some fixed parameters $\gamma\in
(0,1)$, $R>0$ and $\Lambda>0$, and for all $h\in\N$.\\ 

Before going on, we recall the notion of \textit{convergence in the
Kuratowski sense}. Let
$(S_h)_h$ be a sequence of sets in $\R^n$, 
then we say that $S_h$ converges in the Kuratowski sense to a set
$S\subset \R^n$ as $h\to \infty$, if the following two properties
hold:  
\begin{itemize}
\item if a sequence of points $x_h\in S_h$ converges to a point $x$ as
  $h\to \infty$, then $x\in S$; 
\item for any $x\in S$ there exists a sequence $x_h\in S_h$
  such that $x_h$ converges to $x$ as $h\to \infty$. 
\end{itemize}
In addition, given $(S_h)_h$ an equibounded sequence of compact sets,
the convergence of $S_h$ to $S$ in the Kuratowski sense is
equivalent to the convergence in the Hausdorff metric.\\

In the following proposition we recall some crucial properties of
uniform sequences of quasiminimizers (see for instance Theorem 1.9 in  
\cite{Tamanini84}).

\begin{prop}[Properties of quasiminimizers]\label{prop:qmin}
Let $(E_h)_h$ be a uniform sequence of quasiminimizers, i.e. assume
there 
exist $\gamma\in(0,1)$, $R>0$ and $\Lambda>0$ 
such that $E_h \in \QM(\gamma,R,\Lambda)$ for all $h\in \N$. Then, if $E_h$ converges to $E$ in $L^1$, the
following facts hold.
\begin{itemize}
 

\item[(i)] $\de E_h$ converges to $\de E$ in the Kuratowski
  sense, as $h\to \infty$. If in addition $\de E$ is
  compact, then $\de E_h$ converges 
  to $\de E$ in the Hausdorff metric.

\item[(ii)] If $x\in\de^*E$ and $x_h\in\de E_h$ is such that $x_h\to
  x$, then there exists $\bar h$, such that $x_h\in\de^*E_h$ for all
  $h\geq\bar h$. Moreover, $\nu_{E_h}(x_h)\to \nu_{E}(x)$ as $h\to
  \infty$. 
\end{itemize}
\end{prop}

The deviation from minimality (and, thus, the concept of
quasiminimality described above) turns out to be closely related to
another key quantity in De Giorgi's regularity theory: the
\textit{excess}. Given $x\in \R^n$, $r>0$ and $E$ of locally finite
perimeter, the excess of $E$ in $B(x,r)$ is defined as
\begin{eqnarray*}
\excess(E,x,r) &=& r^{1-n} \left(\Per(E,B(x,r)) -
  \left|D\carat E(B(x,r))\right|\right)\\ 
&=& \frac{r^{1-n}}{2} \min_{\stackrel{\xi\in \R^n}{|\xi|=1}} \left\{\int_{\de^* E \cap B(x,r)} |\nu_E(y) - \xi|^2\, d\Hau^{n-1}(y)\right\}.
\end{eqnarray*}

In the following proposition we state a useful continuity property
of the excess and the fundamental regularity
result for quasiminimizers (see for instance Proposition 4.3.1 in
\cite{Ambrosio97} and Theorem 1.9 in \cite{Tamanini84}). Before stating the proposition, we introduce some
extra 
notation. Given a point $x\in \R^n$ and a unit vector $\nu\in \R^n$,
we write with a little abuse of notation $x = \xnp + \xn\nu =
(\xnp,\xn)$, where $\xnp$ is the projection of $x$ onto the orthogonal
complement of $\nu$ and 
$\xn = \langle x,\nu\rangle$. Given $r>0$ and a unit vector $\nu\in
\R^n$, we define the
\textit{cylinder} $\cyl_{\nu,r} = \{x = (\xnp,\xn):\
\max(|\xnp|,|\xn|)<r\}$. Following our notation, $\cyl_{\nu,r}$ can be
defined as the Cartesian product $B_{\nu,r} \times (-r,r)\cdot\nu$,
where $B_{\nu,r}$ is the open ball of radius $r$ in the orthogonal
complement of $\nu$. Given a function $f:B_{\nu,r}\to \R$, we define
its \textit{graph} as
\[
\graph(f) = \{(\xnp, f(\xnp)\nu):\ \xnp\in B_{\nu,r}\}. 
\]

\begin{prop}[Excess and regularity for
  quasiminimizers]\label{prop:psiexc}
Given $\gamma\in (0,1)$, $R>0$ and $\Lambda>0$, the following facts
hold. 
\begin{itemize}
\item[(i)] Let $(E_h)_h$ be a sequence in $\QM(\gamma,R,\Lambda)$ and
  assume $E_h \to E$ in $L^1_{loc}$, as $j\to \infty$. Then 
\[
\lim_j \excess(E_j,x,r) = \excess(E,x,r),
\]
for all $x\in \R^n$ and $0<r<R$ for which $\Per(E,\de B(x,r)) = 0$.

\item[(ii)] There exists $\eps_0 = \eps_0(n, \gamma,R,\Lambda)>0$ with the
following property: if $E\in \QM(\gamma,R,\Lambda)$, $x_0\in \de
E$, and if $\excess(E,x_0,2r)<\eps_0$ for some $0<r<R/2$, then $x_0\in
\de^* E$ and, setting $\nu = \nu_E(x_0)$, one has that 
\[
(\de E - x_0) \cap \cyl_{\nu,r} = \graph(f),
\]
where $f\in C^{1,\gamma}(B_{\nu, r})\to \R$, with $f(0) = |\nabla
f(0)| = 0$. Moreover, one has the H\"older estimate 
    \begin{equation}
      \label{eq:nuholder}
      |\nabla f(v) - \nabla f(w)| \leq C |v-w|^\gamma
    \end{equation}
for all $v,w\in B_{\nu,r}$ and for a suitable
constant $C = C(n, \gamma, R, \Lambda)>0$.
\end{itemize}
\end{prop}
\begin{oss}\label{rmk:regularity}\rm
Given a quasiminimizer $E\in \QM(\gamma,R,\Lambda)$, and owing to
Proposition 
\ref{prop:psiexc} and the fact that for any $x_0\in \de^*E$ one has
$\excess(E,x_0,r) \to 0$ as $r\to 0$, we conclude that $\de^* E$ is
a smooth hypersurface of class $C^{1,\gamma}$. Moreover, by Federer's
blow-up argument (see \cite{Giusti84}), the 
Hausdorff dimension of the \textit{singular set} $\de E\setminus
\de^*E$ cannot exceed $n-8$. Finally, one can show via standard
elliptic estimates for weak solutions to the mean curvature equation
with bounded prescribed curvature (see  
Section 7.7 in \cite{AmbFusPal00}) that, if $E$
is a strong $\Lambda$-minimizer, then $\de^*E$ is
of class $C^{1,\eta}$ for all $0<\eta<1$ (and of class $C^{1,1}$
in dimension $n=2$).  
\end{oss}
\medskip

In what follows we will denote by $\Se^n$ the class of Borel
subsets of $\R^n$ with positive and finite Lebesgue measure. Given  
$E\in\Se^n$, we define its \textit{isoperimetric deficit} $\defP(E)$ and its 
\textit{Fraenkel asymmetry} $\asym(E)$ as follows:
\begin{equation}
  \label{deficit}
  \defP(E) := \frac{\Per(E) - \Per(B_E)}{\Per(B_E)}
\end{equation}
and
\begin{equation}
  \label{asym}
  \asym(E) := \inf\left\{ \frac{|E\difsim (x+B_E)|}{|B_E|},\ x\in \R^n\right\},
\end{equation}
where $B_E$ denotes the ball centered at the origin such that $|B_E|=|E|$
and $E\difsim F$ denotes the symmetric difference of the two sets $E$
and $F$. Since both $\defP(E)$ and
$\asym(E)$ are invariant under isometries and dilations, from now on
we will set $|E|=|B|$ so that $B_E=B$. By definition, the Fraenkel asymmetry $\asym(E)$ satisfies
$\asym(E)\in [0,2)$ and it is zero if and only if $E$ coincides with
$B$ in measure-theoretic sense and up to a translation. Notice that
the infimum in \eqref{asym} is actually a minimum.

\section{The Selection Principle}
\label{sect:selprin}

Given a Borel set $E$ in $\R^n$ with $|E| = |B|$, the classical
isoperimetric inequality states that 
\begin{equation}\label{ISOP}
\Per(E) \geq \Per(B),
\end{equation}
with equality if and only if $\asym(E) = 0$ (i.e., if $E$ coincides
with the ball $B$ up to translations and to negligible sets), that is
to say, the isoperimetric deficit $\defP(E)$ is always non-negative
and zero if and only if $\asym(E) = 0$. 

In the next section  we will provide a new proof of the sharp
quantitative isoperimetric inequality in $\R^n$ which is a
quantitative refinement of \eqref{ISOP} and asserts the 
existence of a positive constant $C$ such that, for any
$E\in\Se^n$ it holds 
\begin{equation}\label{QII_2}
\defP(E)\geq C\asym^2(E).
\end{equation}
With the aim of presenting the main ideas of the method that will lead
to the proof of \eqref{QII_2}, we start with some relatively
elementary comments. 
As we have recalled before, the equality case in the
isoperimetric inequality \eqref{ISOP} is attained precisely when $E$
coincides with a ball in measure-theoretic sense. This uniqueness
property can be 
equivalently stated as the implication 
\begin{equation}
  \label{eq:defasym}
\defP(E)=0\ \impl\ \asym(E)=0.  
\end{equation}
By Lemma 2.3 and Lemma 5.1 in \cite{FusMagPra08} (or via a standard concentration-compactness type argument in \cite{Almgren76} Lemma VI.15) it is possible to strengthen \eqref{eq:defasym} and state the following 
\begin{lemma}\label{lemma:smallness}
For all $\alpha_0>0$ there exists $\delta_0>0$ such that, for any $E\in\Se^n$, if $\defP(E)<\delta_0$ then $\asym(E)<\alpha_0$.
\end{lemma}
It is worth noticing that, as a consequence of Lemma \ref{lemma:smallness}, to prove \eqref{QII_2} it is enough to work in the {\it small asymmetry regime}, i.e. to show that there exist $\alpha_0>0$ and $C_0>0$ such that 
\begin{equation}\label{QII_q}
\frac{\defP(E)}{\asym^2(E)}\geq C_0
\end{equation}
for all $E\in\Se^n$ with $0<\asym(E)<\alpha_0$. In fact, assume
otherwise that $\asym(E)\geq\alpha_0$ and let $\delta_0$ be as in
Lemma \ref{lemma:smallness}. Then, since $\asym(E)<2$, it holds that
$\frac{\defP(E)}{\asym^2(E)}\geq\frac{\delta_0}{4}$, and thus
\eqref{QII_2} follows by taking $C=\min\{C_0,\frac{\delta_0}{4}\}$.

In order to study the small asymmetry regime, it is convenient to introduce the functional $Q:\Se^n\to[0,+\infty]$ defined as
\begin{equation}\label{eq:Q}
  Q(E) = \inf \Big\{ \liminf_k \frac{\defP(F_k)}{\asym(F_k)^2}:\ (F_k)_k\subset\Se^n,\ |F_k|=|E|,\ \asym(F_k)>0,\ |F_k\difsim E|\to 0\Big\}.
\end{equation}
The functional $Q$ is the lower semicontinuous envelope of the quotient $\frac{\defP(E)}{\asym(E)^2}$ with respect to the
$L^1$-convergence of sets and, by the lower semicontinuity of the perimeter and the continuity of the asymmetry with respect to this convergence,
$Q(E)=\frac{\defP(E)}{\asym(E)^2}$ whenever $\asym(E)>0$. Let us now
observe that, by the definition of $Q$, the inequality \eqref{QII_q}
in the small asymmetry regime (and, in turn, \eqref{QII_2}) turns out
to be equivalent to  
\begin{equation}\label{QBm0}
Q(B)>0.
\end{equation}
In order to prove \eqref{QBm0} one may study a recovery sequence for
$Q(B)$, that is a sequence of sets $(W_j)_j$ such that $|W_j|=|B|$,
$\asym(W_j)>0$ and $|W_j\difsim B|\to 0$, for which
$Q(B)=\lim_jQ(W_j)$. However, such a sequence may not be ``good
enough'' to handle in order to get the desired estimate
\eqref{QII_q}. To overcome this problem, we take advantage of the following theorem, which is the main result of this section, and asserts the existence of a recovery sequence $(E_j)_j$ for $Q(B)$ satisfying some useful additional properties which simplify the computation of $Q(B)$. 

\begin{teo}[Selection Principle]\label{teo:SP}
There exists a sequence of sets $(E_j)_j\subset \Se^n$, such that
\begin{itemize}
\item[(i)] $|E_j| = |B|$, $\asym(E_j)>0$ and $\asym(E_j) \to 0$ as $j\to \infty$;

\item[(ii)] $Q(E_j) \to Q(B)$ as $j\to \infty$;

\item[(iii)] for each $j$ there exists a function $u_j\in C^1(\de B)$
  such 
  that $\de E_j = \{(1+u_j(x))x:\ x\in \de B\}$ and $u_j \to 0$ in the
  $C^1$-norm, as $j\to \infty$;

\item[(iv)] $\de E_j$ has (scalar) mean curvature $H_j\in
  L^\infty(\de E_j)$ and $\|H_j - 1\|_{L^\infty(\de E_j)} \to 0$ as $j\to \infty$.
\end{itemize}
\end{teo}

The rest of the section will be devoted to the proof of Theorem
\ref{teo:SP}. The latter will be a consequence of several intermediate
results, most of them having their own independent interest and being
suitable for applications to more general frameworks. The main
ingredients of the proof of Theorem \ref{teo:SP} involve a
penalization argument combined with some properties of quasiminimizers
of the perimeter.

Let $(W_j)_j$ be a recovery sequence for $Q(B)$ having 
\begin{eqnarray}\label{asymbound}
\asym(W_j)\leq\frac{1}{4(Q(B)+2)}
\end{eqnarray}
and satisfying 
\begin{equation}\label{eq:W_j}
  |Q(W_j) -  Q(B)| < \frac{1}{j}\qquad \mbox{for all }j\geq 1.
\end{equation} 
Note that, as pointed out in \cite{Hall92}, by selecting a suitable
sequence of ellipsoids converging to $B$, one can show that $Q(B)<+\infty$ (see also \cite{Maggi08}).\\
We now define the sequence 
of functionals $(Q_j)_j:\Se^n\to[0,+\infty)$ as
\begin{equation}
  \label{eq:Q_j}
  Q_j(E) = Q(E) + \left(\frac{\asym(E)}{\asym(W_j)}-1\right)^2.
\end{equation}
The following lemma holds

\begin{lemma}[Penalization]\label{lemma:exist-approx}
For any integer $j\geq 1$, 
\begin{itemize}
\item[(i)] $Q_j$ is lower semicontinuous with respect to the
  $L^1$-convergence of sets; 

\item[(ii)] there exists a bounded minimizer of the functional $Q_j$,
  i.e. a bounded set $E_j$ such that $|E_j|=|B|$ and $Q_j(E_j) \leq Q_j(F)$ for all $F\in\Se^n$;

\item[(iii)] $E_j \to B$ in $L^1$, $Q(E_j) \to Q(B)$ and
  $\frac{\asym(E_j)}{\asym(W_j)} \to 1$, as $j\to \infty$;

\end{itemize}
\end{lemma}
\begin{proof}
(i) follows from the lower semicontinuity of the
perimeter and the continuity of $\frac{\asym(\cdot)}{\asym(W_j)}$
with respect to $L^1$-convergence of sets.

The proof of (ii) borrows some ideas from Lemma VI.15 in \cite{Almgren76} (see also \cite{Morgan94}). Let $j$ be fixed and let $(V_{j,k})_k\subset\Se^n$ be a minimizing sequence 
for $Q_j$ satisfying $|V_{j,k}| = |B|$, 
$Q_j(V_{j,k}) \leq \inf Q_j + 1/k$, and such that $\asym(V_{j,k})=\frac{|V_{j,k}\difsim B|}{|B|}$ for all $k\geq 1$. Since $\inf Q_j \leq Q_j(W_j) = Q(W_j)$ and
$Q(W_j) \to Q(B)$ as $j\to \infty$, we may assume without loss of
generality that, for all $k\geq 1$, 
\begin{equation}\label{Qjbound}
Q_j(V_{j,k}) \leq Q(B) + 1.
\end{equation}

In particular, this implies that there exists $M>0$ such that $\sup_k\Per(V_{j,k}) \leq M$. By the well-known compactness
properties of sequences of sets with equibounded perimeter, we can
assume that there exists $V_j\in\Se^n$ such that (up to subsequences) $V_{j,k}\to V_j$ in the $L^1_{loc}$ convergence of sets, which in particular implies that $|V_j|\leq\liminf_k|V_{j,k}|=|B|$. Moreover, by the lower semicontinuity of the
perimeter, we have also that $\Per(V_j) \leq M$. By the definition of $Q_j$, thanks to \eqref{asymbound} and \eqref{Qjbound}, we have that
\begin{eqnarray*}
\frac{|V_{j,k}\difsim B|}{|B|}=\asym(V_{j,k})\leq((Q(B)+1)^{\frac
  12}+1)\asym(W_j)\leq\frac{(Q(B)+1)^{\frac 12}+1}{4(Q(B)+2)}<\frac 14.
\end{eqnarray*}
Therefore 
\begin{equation}\label{eq:piudi34}
|V_{j,k} \cap B| > \frac 34 |B|,
\end{equation}
for all $k\in\N$. We now show that
\begin{equation}\label{eq:minper}
\Per(V_j) \leq \Per(F),
\end{equation}
for all sets $F\in \Se^n$ such
that $F\difsim V_j \compact \R^n\setminus B(0,3)$ and $|F| =
|V_j|$. Let us assume by contradiction that \eqref{eq:minper} does not hold, i.e., there exist $\delta>0$ and $F$ as above, such that 
\begin{equation}\label{falsottimo}
\Per(F)\leq\Per(V_j)-\delta.
\end{equation}
Given $0<r<R$, we set $C(r,R)=B(0,R)\setminus\overline {B(0,r)}$ and define $(\hat V_{j,k})_k\subset \Se^n$ as
\begin{eqnarray*}
\hat V_{j,k}=(V_{j,k}\setminus C(r,R))\cup(F\cap C(r,R)).
\end{eqnarray*} 
Note that, by the definition of $F$, by the $L^1_{loc}$ convergence of $V_{j,k}$ to $V_j$ and thanks to \eqref{giusti}, we can choose $r$ and $R$ such that 
\begin{itemize}
\item[(a)] $3<r<R$,
\item[(b)] $F\difsim V_j\subset\subset C(r,R)$,
\item[(c)] $\Hau^{n-1}((V_{j,k}\difsim V_j)\cap\de C(r,R))\to 0$ as $k\to \infty$,
\item[(d)] $\Per (\hat V_{j,k})=\Per(V_{j,k},\R^n\setminus\overline{C(r,R)})+\Per(F,C(r,R))+\Hau^{n-1}((V_{j,k}\difsim V_j)\cap\de C(r,R))$.
\end{itemize}
Let us observe that, since $\Per(V_j,C(r,R))\leq\liminf_k \Per(V_{j,k},C(r,R))$, on combining (c) and (d), and thanks to \eqref{falsottimo}, there exists $k_j\in\N$ such that, for all $k\geq k_j$ we get
\begin{equation}\label{perdelta}
\Per (\hat V_{j,k})\leq \Per(V_{j,k})-\frac{2\delta}3.
\end{equation}
Moreover, by the definition of $\hat V_{j,k}$ we also have that
\begin{eqnarray*}
|\hat V_{j,k}|&=&|F\cap C(r,R)|+|V_{j,k}\setminus C(r,R)|\\&=&|V_{j,k}|+|V_j\cap C(r,R)|-|V_{j,k}\cap C(r,R)|\\&=&|B|+|V_j\cap C(r,R)|-|V_{j,k}\cap C(r,R)|,
\end{eqnarray*}
therefore, passing to the limit as $k\to \infty$, one obtains 
\begin{equation}\label{limvolB}
\lim_k|\hat V_{j,k}|=|B|.
\end{equation}
Let us now fix $x_j\in\de^*F\cap C(r,R)$. Thanks to \eqref{limvolB} and \eqref{density}, for $k$ large enough there exists $0\leq\rho_{j,k}<\left(\frac{\delta}{3n\omega_n}\right)^{\frac{1}{n-1}}$, such that, defining $(\tilde V_{j,k})_k$ as
\begin{eqnarray}\label{falsuc}
\tilde V_{j,k}=\begin{cases}\hat V_{j,k}\cup B(x_j,\rho_{j,k})&\text{if } |\hat V_{j,k}|\leq|B|\cr
\hat V_{j,k}\setminus B(x_j,\rho_{j,k})&\text{if } |\hat V_{j,k}|>|B|,
\end{cases}
\end{eqnarray}
we get $|\tilde V_{j,k}|=|B|$, $B(x_j,\rho_{j,k})\subset\subset C(r,R)$, and 
\begin{equation}\label{perpalla}
|\Per (\hat V_{j,k})-\Per(\tilde V_{j,k})|\leq\Per (B(x_j,\rho_{j,k}))=n\omega_n(\rho_{j,k})^{n-1}<\frac{\delta}{3}.
\end{equation}
By \eqref{perdelta} and \eqref{perpalla}, we eventually get 
\begin{equation}\label{perdeltapalla}
\Per(\tilde V_{j,k})\leq \Per(V_{j,k})-\frac{\delta}{3}.
\end{equation}
This, in turn, would contradict the fact that $V_{j,k}$ is a minimizing sequence for $Q_j$, once we prove that, for $k$ sufficiently large,
\begin{equation}\label{eq:asymFE}
\asym(\tilde V_{j,k}) = \asym(V_{j,k}).  
\end{equation}
Indeed, by \eqref{eq:piudi34} and \eqref{falsuc} we have
\begin{eqnarray}\label{eq:FBpiccolo}
|\tilde V_{j,k}\difsim B| &=& |V_{j,k}\difsim B|\\ \nonumber
&=& 2(|B| - |V_{j,k}\cap B|)\\ \nonumber
&\leq & |B|/2.
\end{eqnarray}
On the other hand, if $x\in \R^n\setminus B(0,2)$ then $V_{j,k}\cap B
\subset \tilde V_{j,k} \difsim (x+B)$, and therefore by \eqref{eq:piudi34} we get 
\begin{equation}
  \label{eq:FBgrande}
|\tilde V_{j,k}\difsim (x+B)| \geq |V_{j,k}\cap B| > \frac 34|B|.
\end{equation}
On combining \eqref{eq:FBpiccolo} and \eqref{eq:FBgrande}, one
shows that the asymmetry of $\tilde V_{j,k}$ is attained on a ball centered in $x\in
B(0,2)$, that is \eqref{eq:asymFE} holds, as wanted.

Thanks to \eqref{eq:minper}, and by well-known results about minimizers of the perimeter
subject to a 
volume constraint, there exists $R>1$ such that $V_j\subset B(0,R)$.

\noindent We now distinguish two cases.

\noindent
\textit{Case $1$.} $|V_j| = |B|$. In this case the local
convergence is equivalent to convergence in $L^1(\R^n)$, hence by the
lower semicontinuity of $Q_j$ we have that $V_j$ is a minimizer of $Q_j$, thus we conclude taking $E_j=V_j$.
 
\noindent
\textit{Case $2$.} $|V_j| < |B|$. In this case the sequence
$(V_{j,k})_k$ ``looses volume at infinity''. We now claim that, setting $x_0 = (R+2,0,\dots,0)\in \R^n$ and $0<t<1$ such 
that $\omega_n t^n + |V_j|= |B|$, the set $E_j := V_j \cup B(x_0,t)$ is a minimizer for $Q_j$. To this
end, note that, since $V_j\subset B(0,R)$, there exists a null set ${\mathcal N}\subset(R,R+1)$ such that, for all $j\geq 1$ and $\rho\in(R,R+1)\setminus{\mathcal N}$, 
we have that
\begin{equation}\label{giusti-2}
\Per(V_{j,k}, B(0,\rho))= \Per(V_{j,k}\cap B(0,\rho))- \Hau^{n-1}(V_{j,k}\cap \de B(0,\rho)),\ \forall k\geq 1,
\end{equation}
thanks to \eqref{giusti}, and
\begin{equation}\label{giusti-3}
\lim_k \Hau^{n-1}(V_{j,k}\cap \de B(0,\rho)) = 0
\end{equation}
since $|V_{j,k}\setminus B(0,\rho)|\to 0$ as $k\to\infty$.

By \eqref{giusti-2} and \eqref{giusti-3}, and owing to the isoperimetric inequality in $\R^n$, we get
\begin{eqnarray}\label{eq:previousineq}
\Per(E_j) &=& \Per(V_j, B(0,\rho)) + n\omega_n t^{n-1}\\ \nonumber
&\leq& \liminf_k \Per(V_{j,k}, B(0,\rho)) + n\omega_n t^{n-1}\\ \nonumber
&=& \liminf_k(\Per(V_{j,k}\cap B(0,\rho))-\Hau^{n-1}(V_{j,k}\cap \de
B(0,\rho)))+n\omega_n t^{n-1}\\ \nonumber 
&=&\liminf_k (\Per(V_{j,k}) - \Per(V_{j,k}\setminus B(0,\rho))) + n\omega_n t^{n-1}\\ \nonumber
&\leq & \liminf_k (\Per(V_{j,k}) - n\omega_n t^{n-1}_k) + n\omega_n
t^{n-1}\\ \nonumber
&=& \liminf_k \Per(V_{j,k}),
\end{eqnarray}
where we have denoted by $t_k$ the radius of a ball equivalent to
$V_{j,k} \setminus B(0,\rho)$ and used the fact that $t_k \to t$ as $k\to
\infty$. Taking into account \eqref{eq:piudi34}, one can check that
the asymmetry of $E_j$ is attained on balls that are disjoint from
$B(x_0,t)$, hence
\begin{equation}\label{eq:asimlimasimk}
0<\asym(E_j) = \lim_k \asym(V_{j,k}).
\end{equation}
Then by \eqref{eq:previousineq} and \eqref{eq:asimlimasimk} we
conclude that $Q_j(E_j) \leq \liminf_k Q_j(V_{j,k}) = \inf Q_j$, as
claimed.  
\medskip


Finally, to prove (iii) we take $E_j$ a minimizer for $Q_j$ and observe that
\[
Q(E_j) \leq Q_j(E_j) \leq Q_j(W_j) = Q(W_j).
\]
This implies that $\asym(E_j) = \asym(W_j) + o(\asym(W_j))$ and that $\lim
Q(E_j) = \lim Q_j(E_j) = Q(B)$. Eventually, by the invariance of $Q_j$ under translation we may
assume that $E_j$ converges to $B$, thus completing the proof.
\end{proof}
We omit the elementary proof of the next lemma. It follows quite directly
from the definition of asymmetry and from the triangular inequality 
\begin{equation}
  \label{trineq}
  |A \difsim B| \leq |A \difsim C| + |C \difsim B|
\end{equation}
which holds in particular for any $A,\ B,\ C\in\Se^n$.
\begin{lemma}\label{lemma:stimasym}
Let $E\in\Se^n$ with $|E|=|B|=\omega_n$. For all $x\in \R^n$ and 
for any $F\in\Se^n$ with $E\difsim F
\compact B(x,\frac 12)$, it holds that
$|\asym(E) - \asym(F)| \leq \frac{2^{n+2}}{(2^n-1)\omega_n} |E\difsim F|$. 
\end{lemma}
We now establish a fundamental property of the sequence $(E_j)_j$ of Lemma \ref{lemma:exist-approx}, i.e., the fact that it is a uniform sequence of $\Lambda$-minimizers (see Section \ref{sect:prelim}).
\begin{lemma}[Uniform $\Lambda$-minimality]\label{lemma:Lambda-min}
There exist $\Lambda = \Lambda(n)>0$ and $j_0\in\N$ with the
following property: for 
all $j\geq j_0$ and for any minimizer $E_j$ of the functional $Q_j$
satisfying $|E_j| = |B|$, 
$Q(E_j) \leq Q(B)+1$, and such that $|\asym(E_j) - \asym(W_j)| \leq \asym(W_j)/2$,
we obtain that $E_j$ is a strong $\Lambda$-minimizer of the perimeter.
\end{lemma}

\begin{proof}
Let $x\in\R^n$ be fixed and let $F\subset\Se^n$ be such that $F\difsim E_j\subset\subset B(x,1/2)$. We want to prove that 
\begin{eqnarray*}
\Per(E_j)\leq\Per(F)+\Lambda\omega_{n-1}\frac{|E_j\difsim F|}{\omega_n}
\end{eqnarray*}
for some $\Lambda=\Lambda(n)>0$. Without loss of generality let us assume that $\Per(F)\leq \Per(E_j)$ and that $\asym(E_j) = \frac{|E_j\difsim
  B|}{|B|}$. We divide the proof in two cases.\\
\noindent
\textit{Case 1.} $\asym(E_j)^2 \leq |E_j \difsim F|$. In this case, 
by the assumption $Q(E_j) \leq Q(B)+1$, we get 
\begin{equation}\label{L1}
  \begin{split}
    \Per(E_j) & \leq \Per(B) + (Q(B)+1)\Per(B) \asym(E_j)^2 \\ 
 &\leq \Per(B)+ (Q(B)+1)\Per(B) |E_j\difsim F|
 \end{split}
 \end{equation}
 By the previous inequality, denoting by $B_F$ the ball equivalent to $F$ centered at the origin, using the isoperimetric inequality in $\R^n$ and the triangular inequality \eqref{trineq} we have 
\begin{equation}\label{L2}
 \begin{split}
\Per(E_j) 
&\leq \Per(F) + \Per(B)-\Per(B_F) + (Q(B)+1)\Per(B) |E_j\difsim F|\\
&\leq \Per(F)+ n\omega_n^{\frac{1}{n}} (|E_j|^{\frac{n-1}{n}} -
|F|^{\frac{n-1}{n}}) + (Q(B)+1)\Per(B) |E_j\difsim F|\\ 
&\leq \Per(F) + n\omega_n^{\frac{1}{n}} ((|F| + |E_j\difsim
F|)^{\frac{n-1}{n}} - |F|^{\frac{n-1}{n}}) + (Q(B)+1)\Per(B) |E_j\difsim F|\\ 
&= \Per(F) + n\omega_n^{\frac{1}{n}} |F|^{\frac{n-1}{n}}((1 + \frac{|E_j\difsim
F|}{|F|})^{\frac{n-1}{n}} - 1) + (Q(B)+1)\Per(B) |E_j\difsim F|.
\end{split}
\end{equation}
Using Bernoulli's inequality and the fact that, by construction, $|F| \geq \frac 34
\omega_n$, by \eqref{L2} we get
\begin{equation}
\begin{split}
\Per(E_j) &\leq \Per(F) + (n-1)\omega_n^{\frac{1}{n}} |F|^{\frac{-1}{n}}
|E_j\difsim F| + (Q(B)+1)\Per(B) |E_j\difsim F|\\ 
&\leq \Per(F) + (n-1) (4/3)^{\frac{1}{n}}
|E_j\difsim F| + (Q(B)+1)\Per(B) |E_j\difsim F|\\ 
& = \Per(F) + \Lambda_1 \omega_{n-1}\frac{|E_j\difsim F|}{\omega_n},
  \end{split}
\end{equation}
where we have set $\Lambda_1=\frac{\omega_n((n-1) (4/3)^{\frac{1}{n}}+(Q(B)+1)\Per(B))}{\omega_{n-1}}$.
\medskip

\noindent
\textit{Case 2.} $|E_j\difsim F| < \asym(E_j)^2$. By the inequality $Q_j(E_j) \leq Q_j(F)$ we obtain
\begin{equation}
  \label{L2.1}
    \defP(E_j) \leq \defP(F) + \left(\frac{\asym(E_j)^2}{\asym(F)^2} -1\right)\defP(F) + \eta,
\end{equation}
where
\begin{equation*}
 \eta := \asym(E_j)^2 \frac{(\asym(F) - \asym(E_j)) (\asym(F) +
    \asym(E_j) - 2\asym(W_j))}{\asym(W_j)^2}. 
\end{equation*}
By noting that the assumption $|\asym(E_j) - \asym(W_j)|\leq \asym(W_j)/2$ implies $\asym(E_j) \leq
3\asym(W_j)/2$, and by exploiting Lemma \ref{lemma:stimasym}, we have that
\begin{equation}\label{stimaA}
\begin{split}
\eta &\leq \tfrac{9}{4} (\asym(F) - \asym(E_j)) (\asym(F) +
    \asym(E_j) - 2\asym(W_j))\\ 
&\leq C_1 |E_j\difsim F|,
\end{split}
\end{equation}
for some $C_1=C_1(n)>0$. By Lemma \ref{lemma:stimasym} we have that
\begin{equation}\label{stimasymfrac1}
\left(\frac{\asym(E_j)^2}{\asym(F)^2} -1\right)\defP(F) \leq \frac{2^{n+4}}{(2^n-1)\omega_n} Q(F) |E_j\difsim F|.
\end{equation}
Observe now that, combining the hypothesis $|E_j\difsim F|<\asym^2(E_j)$ with Lemma \ref{lemma:stimasym} and recalling that $\asym(E_j)\to 0$, we have that there exists $C>0$ and $j_0\in\N$ such that, for all $j\geq j_0$ it holds that
\begin{eqnarray}\label{convergeauno}
\left|\frac{\Per(B)}{\Per(B_F)}-1\right|\leq C\asym(E_j)^2,\quad
\left|\frac{\asym(E_j)}{\asym(F)}-1\right|\leq C\asym(E_j).
\end{eqnarray}  
By the previous estimates, using that, by assumption on $F$, $\Per(F)\leq\Per(E_j)$ we also get that
\[
\begin{split}
Q(F) &\leq \frac{\Per(B) \asym(E_j)^2}{\Per(B_F)\asym(F)^2} Q(E_j) +
\left(\frac{\Per(B)}{\Per(B_F)}-1\right)\frac{1}{\asym(F)^2}.
\end{split}
\]
By the previous inequality, using \eqref{convergeauno}, we have for $j$ large enough 
\[
\begin{split}
Q(F)&\leq 2Q(E_j) + 2\leq 2(Q(B)+1) + 2,
\end{split}
\]
Therefore, \eqref{stimasymfrac1} becomes
\begin{equation}\label{stimasymfrac2}
\left(\frac{\asym(E_j)^2}{\asym(F)^2} -1\right)\defP(F) \leq C_2|E_j\difsim F|,
\end{equation}
with $C_2=C_2(n)>0$. In conclusion, starting from \eqref{L2.1} we have proved that
\[
\defP(E_j) \leq \defP(F) + (C_1+C_2) |E_j\difsim F|,
\]
that is  
\begin{eqnarray*}
\Per(E_j) &\leq & \Per(F)+\left(\frac{\Per(B)}{\Per(B_F)}-1\right)\Per(F) + (C_1+C_2)\Per(B)|E_j\difsim F|\\
&\leq&\Per(F)+\Lambda_2\Per(B)|E_j\difsim F|,
\end{eqnarray*}
with $\Lambda_2=(C_1+C_2)\Per(B)+1$. 
\medskip

The conclusion follows by setting $\Lambda =
\max(\Lambda_1,\Lambda_2)$. 
\end{proof}

In the next lemma, we prove the $C^{1,\gamma}$ regularity of $\de E_j$
for $j$ large enough, as well as the fact that $\de E_j$
converges to $\de B$ in the $C^1$-topology, as $j\to \infty$. Here, by convergence of $\de E_j$
to $\de B$ in the $C^1$-topology, we mean the following: there exist $r>0$ and an open
covering of $\de B$ by a finite family of cylinders
$\{\nu_k + \cyl_{\nu_k,r}\}_{k=1}^N$, with $\nu_k \in \de B$ 
such that it holds
\begin{itemize}
\item $\de E_j\subset\bigcup\limits_{k=1}^N(\nu_k + \cyl_{\nu_k,r})$ for $j$ large;
\item $\de E_j\cap \cyl_{\nu_k,r}=\graph({g_{j,k}})$ for some function $g_{j,k}\in C^1(B_{\nu_k,r})$, $k=1,\dots,N$, and for $j$ large;
\item $g_{j,k}\to g_k$ in $C^1$ as $j\to\infty$, where $g_k\in C^1(B_{\nu_k,r})$ is such that $\de B\cap \cyl_{\nu_k,r}=\graph({g_{k}})$, for $k=1,\dots,N$. 
\end{itemize}
\begin{lemma}[Regularity]\label{lemma:regconv}
There exists $j_1\in \N$ such that, for all $j\geq j_1$ and for any
minimizer $E_j$ of $Q_j$, $\de E_j$ is of class
$C^{1,\eta}$ for any $\eta\in (0,1)$. Moreover, $\de E_j$
converges to $\de B$ in the $C^1$-topology, as $j\to \infty$. 
\end{lemma}
\begin{proof}
First, we set $e_n = (0,\dots,0,1)\in \R^n$ and for a given $x\in
\R^n$ we write $x = (x',x_n) = (\xenp,\xen)$ following the notation
introduced in Section \ref{sect:prelim}. For a given $r>0$ we set 
\[
A_r = \{x'\in \R^{n-1}:\ |x'|<r\}.
\] 
We recall that, owing to Lemma \ref{lemma:Lambda-min} and for $j\geq
j_0$, $E_j\in \QM(\frac 12,\frac 12,\Lambda)$. Then, recalling the above definition of $C^1$ convergence of smooth boundaries, 
it is enough to prove that there exists $j_1\geq j_0$ and
a small $r_1>0$, such that one can 
find a sequence of functions $(g_j)_{j}$, with $g_j \in
C^{1,\frac 12}(A_{r_1})$ for all $j\geq j_1$, and satisfying the
following two properties:  
\begin{equation}
  \label{eq:graphgj}
(\de E_j - e_n)  \cap \cyl_{e_n,r_1} = \graph(g_j)\quad
\forall\, j\geq j_1,
\end{equation}
where $\cyl_{e_n,r_1} = A_{r_1} \times (-r_1,r_1)$;
\begin{equation}
  \label{eq:convgj}
  \|g_j - g\|_{C^1(A_{r_1})} \to 0\quad \text{as }j\to \infty,
\end{equation}
where we have set $g(x') = \sqrt{1-|x'|^2} -1$. 
Then, the proof of the
lemma will be completed on taking into account Remark
\ref{rmk:regularity}.\\ 

To prove \eqref{eq:graphgj} and \eqref{eq:convgj} above, we choose
$0<r<1$ such that $\excess(B,e_n,4r) < \frac{\eps_0}{2^{n-1}}$, where
$\eps_0$ 
is as in Proposition \ref{prop:psiexc} (ii) relative to $\QM(\frac 12,\frac 12,\Lambda)$. Thanks to Propositions
\ref{prop:psiexc} (i) and \ref{prop:qmin} (i)--(ii), we can
find $j_1\in \N$ such that for all $j\geq j_1$
\begin{itemize}
\item[(a)] $\de E_j\cap B(e_n,r)\neq \emptyset$,
\item[(b)] $\excess(E_j,e_n,4r)<\frac{\eps_0}{2^{n-1}}$,
\item[(c)] there exists $x_j\in \de^* E_j\cap B(e_n,r)$ such that
  $x_j\to e_n$ and $\nu_j:= \nu_{E_j}(x_j)\to e_n$.
\end{itemize}
By the 
definition of the excess, by the inclusion $B(x_j,2r) \subset
B(e_n,4r)$, and by (b) above, we have 
\begin{eqnarray*}
\excess(E_j,x_j,2r) &=& \frac{(2r)^{1-n}}{2} \inf_{|\xi|=1} \int_{\de^* E_j \cap B(x_j,2r)}
|\nu_{E_j}(z) - \xi|^2\, d\Hau^{n-1}(z)\\ 
&\leq & \frac{(2r)^{1-n}}{2} \int_{\de^* E_j \cap B(x_j,2r)}
|\nu_{E_j}(z) - \xi_j|^2\, d\Hau^{n-1}(z)\\
&\leq & \frac{(2r)^{1-n}}{2} \int_{\de^* E_j \cap B(e_n,4r)}
|\nu_{E_j}(z) - \xi_j|^2\, d\Hau^{n-1}(z)\\ 
&=& 2^{n-1} \excess(E_j,e_n,4r)\\ 
&<& \eps_0
\end{eqnarray*}
for $\xi_j =
\frac{D\carat{E_j}(B(e_n,4r))}{|D\carat{E_j}|(B(e_n,4r))}$ and for all
$j\geq j_0$. 
Thanks to Lemma \ref{lemma:Lambda-min} and 
Proposition \ref{prop:psiexc} (ii), 
there exists a sequence of functions $f_j\in
C^{1,\frac 12}(B_{\nu_j,r})$, such that $f_j(0) = |\nabla f_j(0)| = 0$ and 
$(\de E_j -x_j) \cap \cyl_{\nu_j,r} = \graph(f_j)$. At this point, one
can check that, setting $r_1 = r/2$ and taking a larger $j_1$ if needed, the following facts hold:
\begin{itemize}
\item[(d)] $\cyl_{e_n,r_1} \subset \cyl_{\nu_j,r}$ for $j\geq j_1$,

\item[(e)] we can find $g_j\in C^{1,\frac 12}(A_{r_1})$ for
  $j\geq j_1$ such that $\graph(g_j) = (x_j-e_n + \graph(f_j)) \cap \cyl_{e_n,r_1}$,

\item[(f)] 
$\|g_j - g\|_{L^\infty(A_{r_1})} \to 0$, where $g(x') =
  \sqrt{1-|x'|^2} -1$.

\end{itemize}
Indeed, (d) is a direct consequence of Proposition \ref{prop:qmin} (ii). Then,  
(e) follows on recalling that $x_j \to e_n$ by (c) and that $\nabla f_j$ is 
$\frac 12$-H\"older continuous (uniformly in $j$), thanks to 
\eqref{eq:nuholder}. Finally, (f) can be proved on using (c) and 
Proposition \ref{prop:qmin} (i).
  
Owing to (e) above and to the properties of $f_j$, we obtain
\eqref{eq:graphgj}. Then, thanks to (d), (e) and (f) combined with
\eqref{eq:nuholder}, we get 
\begin{equation}\label{gholder}
|\nabla g_j(v) - \nabla g_j(w)| \leq C |v-w|^{\frac 12}
\end{equation}
for all $v,w\in A_{r_1}$ and for a constant $C>0$ independent of
$j$. By a contradiction argument using (iii), \eqref{gholder}, and
Ascoli-Arzel\`a's Theorem, we finally conclude that 
\[
\|g_j - g\|_{C^1(A_{r_1})} \to 0
\]
as $j\to\infty$, thus proving \eqref{eq:convgj}. This
completes the proof of the lemma. 
\end{proof} 
\medskip

In the following lemma, we show that the (scalar) mean curvature $H_j$
of $\de E_j$ is in $L^\infty(\de E_j)$. Then, we compute a first
variation inequality of $Q_j$ at $E_j$ that translates into a
quantitative estimate of the oscillation of the mean curvature.
\begin{lemma}\label{lemma:firstvar}
Let $j\geq j_1$, with
$j_1$ as in Lemma \ref{lemma:regconv}, and let $E_j$ be a minimizer of $Q_j$. Then 
\begin{itemize}
\item[(i)] $\de E_j$ has scalar mean curvature $H_j \in
  L^\infty(\partial E_j)$ (with
orientation induced by the inner normal to $E_j$). Moreover, for
${\Hau}^{n-1}$-a.e. $x,y\in \de E_j$, one has 
\begin{equation}
  \label{firstvar}
  |H_j(x) - H_j(y)| \leq \frac{2n}{n-1} \left( Q(E_j)
    \asym(E_j) + \frac{\asym(E_j)^2}{\asym(W_j)^2} |\asym(E_j) -
    \asym(W_j)|\right);
\end{equation}

\item[(ii)] $\lim_j\|H_j-1\|_{L^\infty(\de E_j)}=0$.
\end{itemize}
\end{lemma}
\begin{proof}
To prove the theorem we consider a ``parametric
inflation-deflation'', that will lead to the first variation
inequality \eqref{firstvar} and, in turn, to (ii). 

Let us fix $x_1,\ x_2\in\partial E_j$ such that $x_1\not=x_2$. By
Lemma \ref{lemma:regconv}, for $j\geq j_1$ there exist $r>0$, two unit
vectors $\nu_1, \nu_2\in \R^n$, and two functions 
$f_1\in C^1(B_{\nu_1,r})$ and $f_2\in C^1(B_{\nu_2,r})$, 
such that 
$(x_1 + \cyl_{\nu_1,r})\cap (x_2 + \cyl_{\nu_2,r})=\emptyset$ and
\begin{eqnarray*}
(\de E_j - x_m) \cap \cyl_{\nu_m,r} = \graph(f_m),\quad m=1,2.
\end{eqnarray*}  
For $m=1,2$ we take $\varphi_m\in C^1_c(B_{\nu_m,r})$ such that
$\varphi_m\geq 0$ and 
\begin{equation}
  \label{eq:fipsi}
  \int_{B_{\nu_m,r}} \varphi_m = 1.
\end{equation}
Let $\eps>0$ be such that, setting $f_{m,t}(w) = f_m(w) +
t\varphi_m(w)$ for $w\in B_{\nu_m,r}$, one has $\graph(f_{m,t}) \subset
\cyl_{\nu_m,r}$ for all $t\in (-\eps,\eps)$. We use the functions
$f_{m,t}$, $m=1,2$, to modify the set 
$E_j$, i.e. we define 
\[
E_{j,t} = \Big(E_j \setminus \bigcup_{m=1,2} (x_m +
\cyl_{\nu_m,r})\Big) \cup \big(x_1 + \sgr(f_{1,t})\big)\cup \big(x_2 + \sgr(f_{2,-t})\big),
\]
where
\[
\sgr(f_{m,s}) = \{(w,l)\in\cyl_{\nu_m,r}:\ l<f_{m,s}(w)\}.
\]
By \eqref{eq:fipsi} one immediately deduces that $|E_{j,t}| =
|E_j|$. Moreover, by a standard computation one obtains
\begin{equation}\label{teo:firstvar_perimeter}
\frac{1}{n-1}\frac{d}{dt}P(E_{j,t})_{|_{t=0}} =
\int_{B_{\nu_1,r}}h_1 \varphi_1 -
\int_{B_{\nu_2,r}}h_2 \varphi_2,
\end{equation}
where for $m=1,2$
\[
h_m(v) := H_j(v,f_m(v)) = -\frac{1}{n-1}\mdiv \left( \frac{\nabla
f_m(v)}{\sqrt{1+|\nabla f_m(v)|^2}}\right).
\]
Then, by Theorem 4.7.4 in \cite{Ambrosio97}, the
$L^\infty$-norm of $H_j$ over $\de E_j$ turns out to be bounded by the
constant $4\Lambda/(n-1)$. 

By the definition of $E_{j,t}$ one can verify that, for $t>0$  
\begin{eqnarray}\label{teo:firstvar_asymmetry}
|\asym(E_{j,t})-\asym(E_j)|\leq \frac{t}{\omega_n}.
\end{eqnarray}
By \eqref{teo:firstvar_perimeter} and \eqref{teo:firstvar_asymmetry},
and for $t>0$, we also have that
\begin{eqnarray*}
Q(E_{j,t})&=&\frac{P(E_{j,t})-P(B)}{P(B)\asym(E_{j,t})^2} \leq
\frac{P(E_{j,t})-P(B)}{P(B)}\cdot\frac{1}{\asym(E_j)^2\left(
1-\frac{t}{\asym(E_j)\omega_n }\right)^2}\\&\leq& 
Q(E_j)\cdot\frac{1}{1-\frac{2t}{\asym(E_j)\omega_n}} +
\frac{t}{P(B)\asym(E_j)^2}\frac{d}{dt}P(E_{j,t})_{|t=0} + o(t)\\  
&\leq& Q(E_j)+\frac{t}{\omega_n\asym(E_j)}\left(2Q(E_j) +
  \frac{1}{n\asym(E_j)}\frac{d}{dt}P(E_{j,t})_{|t=0}\right) + o(t). 
\end{eqnarray*}
On using again \eqref{teo:firstvar_asymmetry}, we get 
\begin{eqnarray*}
Q_j(E_{j,t})&\leq& Q_j(E_j)+\frac{t}{\omega_n\asym(E_j)}\left(2Q(E_j)
  + \frac{1}{n\asym(E_j)}\frac{d}{dt}P(E_{j,t})_{|t=0}\right)\\
& &\quad +\frac{2t}{\omega_n\asym(W_j)^2}|\asym(E_j)-\asym(W_j)| + o(t).
\end{eqnarray*}
Exploiting now the minimality hypothesis $Q_j(E_j)\leq Q_j(E_{j,t})$
in the previous inequality, dividing by $t>0$, multiplying by
$n\omega_n\asym(E_j)^2$, and finally taking the limit as
$t$ tends to $0$, we obtain 
\begin{equation}\label{ttendeazero}
0\leq 2nQ(E_j)\asym(E_j) +
  \frac{d}{dt}P(E_{j,t})_{|t=0} +
2n\frac{\asym(E_j)^2}{\asym(W_j)^2}|\asym(E_j)-\asym(W_j)|.    
\end{equation}
Let now $w_m\in B_{\nu_m,r}$ be a Lebesgue
point for $h_{f_m}$, $m=1,2$. On choosing a sequence
$(\varphi_m^k)_k\subset C^1_c(B_{\nu_m,r})$ of non-negative
mollifiers, such that  
\[
\lim_k \int_{B_{\nu_m,r}} h_{f_m} \varphi_m^k = h_{f_m}(w_m)
\]
for $m=1,2$, we obtain that for $E_{j,t}^k$ defined as before, but
with $\varphi_m^k$ replacing $\varphi_m$, it holds
\begin{eqnarray}\label{eq:firstvar_curvature}
\frac{1}{n-1}\lim_{k}\frac{d}{dt}P(E_{j,t}^k)_{|_{t=0}} &=& \lim_k
\int_{B_{\nu_1,r}} h_{f_1}\varphi_1^k - \int_{B_{\nu_1,r}} h_{f_2}\varphi_2^k\\ 
\nonumber &=& h_{f_1}(w_1) - h_{f_2}(w_2).
\end{eqnarray}
Moreover, from \eqref{ttendeazero} with $E_{j,t}^k$ in place of
$E_{j,t}$ and thanks to \eqref{eq:firstvar_curvature}, we get 
\begin{equation}
h_{f_2}(w_2) - h_{f_1}(w_1) = -\frac{1}{n-1}\lim_k
\frac{d}{dt}P(E_{j,t}^k)_{|t=0} \leq 
\frac{2n}{n-1}\left(Q(E_j)\asym(E_j) +  
 \frac{\asym(E_j)^2}{\asym(W_j)^2}|\asym(E_j)-\asym(W_j)|\right).    
\end{equation}
The proof of \eqref{firstvar}, and therefore of claim (i), is achieved
by exchanging the roles of $x_1$ and $x_2$.
\medskip

\medskip

Finally, to prove $(ii)$ we recall that $\sup_j\|H_j\|_{L^\infty(\de E_j)}\leq
4\Lambda/(n-1)$. Moreover, by \eqref{firstvar} we have that
\begin{eqnarray}\label{eq:firstvar_osc}
\lim_j\ \esssup_{x,y\in\partial E_j}|H_j(x)-H_j(y)|= 0.
\end{eqnarray}
Thanks to \eqref{eq:firstvar_osc} we conclude that, up to
subsequences, there exists a constant $H$ such that $\|H_j-H\|_{L^{\infty}(\de E_j)}\to 0$ as $j\to
\infty$. By Lemma \ref{lemma:regconv}, $\de E_j$ converges
to $\de B$ in $C^1$ and thus we can consider $U = B_{e_n,\frac 12} \subset
\R^{n-1}$ such that, for $j$ large enough, the portion of the boundary
of $E_j$ inside the open set $U\times (0,+\infty)e_n
\subset \R^n$ is
the graph of a function $f_j\in C^{1}(U)$ converging to the function
$f(w) = \sqrt{1 - |w|^2}$ in the $C^1$-norm, as $j\to \infty$. As a
consequence, adopting the Cartesian notation for the mean curvature as
in (i), 
\begin{eqnarray*}
\lim_j\int_{U}h_{f_j}\varphi &=& \lim_j\int_{U}\frac{\langle\nabla
  f_j,\nabla\varphi\rangle}{\sqrt{1+|\nabla f_j|^2}}\\ 
&= & \int_{U}\frac{\langle\nabla
  f,\nabla\varphi\rangle}{\sqrt{1+|\nabla f|^2}}\\ 
&=& \int_{U}h_{f}\varphi,
\end{eqnarray*}
for any $\varphi\in C^1_c(U)$. This proves that
$H$ coincides with the mean curvature of the ball $B$, i.e. $H=h_f =
1$. It is then easy
to conclude that the whole sequence $H_j$ must converge to $H=1$, and
this completes the proof of (ii).
\end{proof}

\begin{proof}[Proof of Theorem \ref{teo:SP}]
Statements (i) and (ii) of the thesis follow by Lemma \ref{lemma:exist-approx}. The proof of (iii) is an elementary consequence of Lemma \ref{lemma:regconv}, while (iv) follows by Lemma \ref{lemma:firstvar}.
\end{proof}

\section{Two applications of the Selection Principle}\label{sect:app}

In this section we
describe two applications of Theorem
\ref{teo:SP}. The first one is a new proof of the
sharp quantitative isoperimetric inequality in $\R^n$. The second one
is a positive answer to a conjecture by Hall \cite{Hall92}, 
concerning the optimal asymptotic constant for \eqref{QII} in $\R^2$,
when the asymmetry vanishes.
\medskip

\subsection{The Sharp Quantitative Isoperimetric Inequality}

We start by recalling the definition of \textit{nearly spherical set}
introduced by Fuglede in \cite{Fuglede89} (see also \cite{Fuglede86}).
A Borel set $E$ in $\R^n$ is nearly spherical if $|E| = |B|$, the barycenter of $E$ is $0$, and $\de E$ is the
normal graph of a Lipschitz function $u:\de B \to (-1,+\infty)$ 
(i.e., $\de E = \{(1+ u(x))x:\ x\in \de B\}$) with
$\|u\|_{L^\infty(\de B)} \leq \frac{1}{20n}$ and $\|\nabla
u\|_{L^\infty(\de B)}\leq \frac 12$.
In \cite{Fuglede89} (see also \cite{Fuglede86} for a proof in
dimension $2$ and $3$) Fuglede proved the following crucial estimate,
whence the sharp quantitative isoperimetric
inequality easily follows: 
\begin{teo}[Fuglede's estimate]\label{teo:fuglede}
Let $E\subset\R^n$ be a nearly spherical set with $\de E = \{(1+ u(x))x:\ x\in \de B\}$ and $u\in
  W^{1,\infty}(\de B)$ as above, then there
exists $C=C(n)>0$ such that  
\begin{equation*}
\defP(E)\geq C\|u\|^2_{W^{1,2}(\de B)}.
\end{equation*} 
\end{teo}
By appealing to the Selection Principle and to the estimate above, we
could directly provide the complete proof of the sharp quantitative
isoperimetric inequality (see the proof of Theorem
\ref{teo:QII}). Instead, for the sake of completeness, in Lemma
\ref{lemma:fuglede-weak} we provide 
a slightly weaker, but still sufficient for our aims, version of Theorem \ref{teo:fuglede}. Note that its proof is obtained by
basically repeating the argument exploited by Fuglede in the proof of
Theorem \ref{teo:fuglede}.  

Let us first recall the following facts. Let $E\subset\R^n$ be such
that $\de E=\{(1+u(x)x),\ x\in\de B)\}$ for some $u:\de B\to(-1,+\infty)$ of class $C^1$, then the perimeter $\Per(E)$, the Lebesgue measure $|E|$ and the barycenter $\bary(E)$ of $E$ can be computed by exploiting the following formulas:
\begin{equation}
  \label{eq:per}
  \frac{\Per(E)}{\Per(B)} = \int_{\de B} (1+u)^{n-1} \sqrt{1 +
    (1+u)^{-2} |\nabla u|^2}\, d\sigma,
\end{equation}
\begin{equation}
  \label{eq:vol}
  \frac{|E|}{|B|} = \int_{\de B} (1+u)^n\, d\sigma,
\end{equation}
and 
\begin{equation}
  \label{eq:bar}
  \bary(E) = \int_{\de B} (1+u(x))^{n+1}\,x\,\, d\sigma(x),
\end{equation}
where we have set $\sigma =\frac{1}{\Per(B)}\Hau^{n-1}$.

\begin{lemma}[Weak form of Fuglede's estimate]\label{lemma:fuglede-weak}
Let $E\subset\R^n$ and $u:\de B\to(-1,+\infty)$ of class $C^1$ be such that $\de E=\{(1+u(x))x,\ x\in\de B)\}$, $|E|=|B|$ and $\bary(E)=0$. Then for all $\eta>0$ there exists $\e>0$ such that, if $\|u\|_{L^\infty(\de B)} + \|\nabla u\|_{L^\infty(\de B)} <\e$, it holds that
  \begin{equation}
    \label{eq:stimaFuglede}
    \defP(E) \geq \frac{(1-\eta)}2 \|u\|^2_{L^2(\de B)} + \frac 14 \|\nabla u\|^2_{L^2(\de B)}.
  \end{equation}
\end{lemma}
\begin{proof}
By applying Taylor's formula in \eqref{eq:per}, and thanks to the
bound on the sum of the $L^\infty$-norms of $u$ and $\nabla u$, we have that
\begin{equation}
  \label{eq:taylor1}
  \begin{split}
\frac{\Per(E)}{\Per(B)}= \int_{\de B}\left(1 + \frac{|\nabla
    u|^2}{2} \right.& \left.+ (n-1)u + \frac{(n-1)(n-2)}{2} u^2\right) \ d\sigma \\ 
  & + O(\e)(\|u\|^2_{L^2(\de B)}+\|\nabla u\|^2_{L^2(\de B)}).    
  \end{split}
\end{equation}
By the hypothesis $|E| = |B|$, which is equivalent to 
$\int_{\de B} ((1+u)^n - 1)\ d\sigma = 0$, it turns out, again by Taylor's formula, that
\begin{equation}
  \label{eq:taylor2}
  \int_{\de B} u \ d\sigma= 
  -\left(\frac{n-1}{2} + O(\e)\right) \|u\|^2_{L^2(\de B)}.
\end{equation}
Combining \eqref{eq:taylor1} and \eqref{eq:taylor2} we get
\begin{equation}
  \label{eq:stima1}
  \begin{split}
    \defP(E)
&= \int_{\de B} \left(\frac{|\nabla
    u|^2}{2} + (n-1)u + \frac{(n-1)(n-2)}{2} u^2\right)\ d\sigma + O(\e)(\|u\|^2_{L^2(\de B)}+\|\nabla u\|^2_{L^2(\de B)})\\ 
&= \frac 12 \int_{\de B} (|\nabla u|^2 - (n-1)u^2)\ d\sigma + O(\e)(\|u\|^2_{L^2(\de B)}+\|\nabla u\|^2_{L^2(\de B)}).
  \end{split}
\end{equation}
Thanks to the previous estimate, in order to prove the thesis it is only left to prove that, for all $\eta>0$  
\begin{equation}
  \label{eq:stima2}
\|\nabla u\|^2_{L^2(\de B)} - (n-1)\|u\|^2_{L^2(\de B)} \geq (1-\eta)\|u\|^2_{L^2(\de B)} + \frac 12 \|\nabla u\|^2_{L^2(\de B)}  
\end{equation}
for $\e>0$ small enough. To this end, it will be sufficient to consider the Fourier series of $u$ over the orthonormal basis of
spherical harmonics $\{Y_k:\ k=0,1,\dots\}$, namely
\begin{equation*}
u = \sum_{k=0}^\infty a_k Y_k,
\end{equation*}
and estimate the first two coefficients $a_0$ and $a_1$. We start by recalling that 
\begin{equation} \label{eq:Y0Y1}
  Y_0 = 1,\qquad Y_1(x) = x\cdot \nu
\end{equation}
for a suitably chosen $\nu\in \R^n$. Thus the first two coefficients
$a_0, a_1$ of the Fourier expansion of $u$ are given by
\begin{equation*}
a_0 = \int_{\de B} u Y_0\, d\sigma= \int_{\de B} u\, d\sigma \quad\text{and}\quad a_1 = \int_{\de B} u Y_1\, d\sigma = \int_{\de B} u(x) x\cdot \nu\, d\sigma.
\end{equation*}
We first estimate $a_0$. Taking into account that $\|u\|_{L^\infty(\de
B)}<\e$, we have that
\begin{equation}
  \label{eq:est0}
a_0^2 = O(\e^2) \|u\|^2_{L^2(\de B)}.
\end{equation}
We now estimate $a_1$. Observing that, by \eqref{eq:Y0Y1} and by the hypothesis $\bary(E)=0$ 
\[
\int_{\de B} Y_1\, d\sigma = 0,
\]
and that 
\[
\int_{\de B} (1+u)^{n+1} Y_1\, d\sigma = \bary(E) \cdot \nu = 0
\]
we first obtain that
\begin{equation}
  \label{eq:stimaY1}
\int_{\de B} ((1+u)^{n+1} - 1)Y_1\, d\sigma = \int_{\de B} \left((n+1)u + \sum_{k=2}^{n+1}
\binom{n+1}{k} u^k\right)Y_1\, d\sigma = 0.
\end{equation}
Then, from \eqref{eq:stimaY1} we derive 
\[
a_1 = \int_{\de B} uY_1\, d\sigma = -\sum_{k=2}^{n+1}\binom{n}{k} \int_{\de B} u^k\, d\sigma = O(\|u\|^2_{L^2(\de B)})
\]
and
\begin{equation}
  \label{eq:est1}
a_1^2 = O(\e^2) \|u\|^2_{L^2(\de B)}.  
\end{equation}
Since 
\begin{equation}
 \label{eq:u2g2}
 \|u\|^2_{L^2(\de B)} = \sum_{k=0}^\infty a_k^2\quad \text{and}\quad \|\nabla u\|^2_{L^2(\de B)} =
\sum_{k=1}^\infty \lambda_k a_k^2,
\end{equation}
where 
\begin{equation}
  \label{eq:lambdak}
\lambda_k = k(k+n-2)  
\end{equation}
denotes the $k$-th eigenvalue of the Laplace-Beltrami operator on $\de
B$ (relative to the $k$-th eigenfunction 
$Y_k$), on gathering together \eqref{eq:est0} and \eqref{eq:est1} we obtain
\begin{equation}
  \label{eq:est2}
\|u\|^2_{L^2(\de B)} \leq (1+O(\e^2))\sum_{k=2}^\infty a_k^2,\qquad \|\nabla u\|^2_{L^2(\de B)}
\leq (1+O(\e^2))\sum_{k=2}^\infty \lambda_k a_k^2.
\end{equation}
As a consequence of \eqref{eq:lambdak} and \eqref{eq:est2}, the
left-hand side of \eqref{eq:stima2} can be estimated as 
follows:
\begin{equation}
  \label{eq:stima3}
  \begin{split}
  \int_{\de B} (|\nabla u|^2 - (n-1)u^2)\, d\sigma &= \sum_{k=2}^\infty (\lambda_k - n+1)
  a_k^2 + O(\e^2) \|u\|^2_{L^2(_{\de B})}\\ 
&= \frac 12 \|\nabla u\|^2_{L^2(_{\de B})} + \sum_{k=2}^\infty (\lambda_k/2 - n+1)
  a_k^2 + O(\e^2) \|u\|^2_{L^2(_{\de B})}\\ 
&\geq \frac 12 \|\nabla u\|^2_{L^2(_{\de B})} + (1-O(\e^2))\|u\|^2_{L^2(_{\de B})},
  \end{split}
\end{equation}
which in turns imply the desired estimate \eqref{eq:stima2}, and hence
the thesis of the lemma.  
\end{proof}

We are now ready to prove the main result of the section:

\begin{teo}[The Sharp Quantitative Isoperimetric Inequality in $\R^n$]\label{teo:QII}
There exists a positive constant $C$ such that, for any $E\in\Se^n$ it holds
\begin{equation}\label{QII_3}
\defP(E)\geq C\asym(E)^2.
\end{equation}
\end{teo}
\begin{proof}

We claim that 
\begin{equation}\label{eq:QBpositive}
Q(B)>0.
\end{equation}
Suppose the claim proved. Then, by definition of $Q(B)$, there exists
$\alpha_0>0$ such that,  
for all $E\in\Se^n$ with $\asym(E)<\alpha_0$ it holds that 
\begin{equation}\label{teo:stima1}
Q(E)\geq \frac{Q(B)}{2}.
\end{equation}
If otherwise $E$ is such that $\alpha_0\leq\asym(E)<2$, then by Lemma \ref{lemma:smallness} there exists $\delta_0>0$ such that $\defP(E)\geq\delta_0$, 
which implies
\begin{equation}\label{teo:stima2}
Q(E)=\frac{\defP(E)}{\asym(E)^2}\geq \frac{\delta_0}{4}.
\end{equation}
On combining \eqref{teo:stima1} and \eqref{teo:stima2}, we obtain \eqref{QII_3} by choosing $C=\min\{\frac{Q(B)}{2},\frac{\delta_0}{4}\}$.

We are thus left with the proof of \eqref{eq:QBpositive}. To compute
$Q(B)$, we will exploit the sequence $(E_j)_j\subset\Se^n$ provided by
the Selection Principle (Theorem \ref{teo:SP}). Since  $\bary(E_j)\to
0$ as $j\to\infty$, without loss of generality (that is, up to replacing $E_j$ by the sequence $E_j-\bary(E_j)$) we may suppose that $E_j$ fulfills the hypotheses of Lemma \ref{lemma:fuglede-weak}. This in particular implies the existence of $j_0>0$ such that, for all $j\geq j_0$
\begin{equation}
    \label{eq:stimaFuglede-1}
    \defP(E_j) \geq \frac{1}4 \|u_j\|^2_{L^{2}(\de B)}\ .
\end{equation}
By applying Bernoulli and H\"older inequalities, we find two positive
dimensional constants $c_0$ and $c_1$ such that 
\begin{eqnarray*}
\asym(E_j)^2\leq 4|E_j\setminus B|^2&\leq& c_0\left(\int_{\de B\cap(E\setminus B)}(1+u_j)^n-1\, d\Hau^{n-1}\right)^2\\
&\leq& c_0\ n^2\left(\int_{\de B}|u_j|\, d\Hau^{n-1}\right)^2\leq c_1\|u_j\|^2_{L^2(\de B)},
\end{eqnarray*}
Therefore, the estimate \eqref{eq:stimaFuglede-1} implies that, for all $j\geq j_0$,
$$
Q(E_j)\geq C>0
$$
with $C=\frac{1}{4c_1}>0$. The claim then follows, on passing to the
limit in the left hand side of the previous inequality and on recalling
that $Q(E_j)\to Q(B)$ as $j\to\infty$.
\end{proof}

\begin{oss}\rm
It is worth noticing that, by the definition of $Q(B)$, for any $E\in\Se^n$ the following estimate holds true:
\begin{equation*}
\defP(E)\geq Q(B)\asym(E)^2+o(\asym(E)^2).
\end{equation*}
In other words, $Q(B)$ is the best (asymptotic) constant in the sharp
  isoperimetric inequality in $\R^n$, as the asymmetry converges to
  zero. Thanks to the Selection Principle, the computation of $Q(B)$
  becomes an affordable task, as the class of sets which actually
  play a role in this computation is quite restricted. In the next subsection it will be shown that in $\R^2$ it holds $Q(B)=\frac{\pi}{8(4-\pi)}$.
\end{oss}

\subsection{The best constant for the quantitative isoperimetric inequality in $\R^2$ in the small asymmetry regime}

In this section, and more precisely in Theorem \ref{teo:Hconj}, we positively answer to a conjecture posed by Hall in \cite{Hall92} asserting that, for any measurable set in $\R^2$ with positive and finite Lebesgue measure, the following estimate holds true:
\begin{equation}
  \label{eq:HHWconj-2}
 \defP(E) \geq C_0 \asym(E)^2 + o(\asym(E))^2,
\end{equation}
with $C_0=\frac{\pi}{8(4-\pi)}$ optimal. \\
We start by recalling a result conjectured in \cite{HalHayWei91} Section V, and proved in \cite{HalHay93} Theorem 1:
\begin{teo}[Hall-Hayman-Weitsmann]\label{teo:HHW91} Let $E\in\Se^2$ be a
  convex set, then \eqref{eq:HHWconj-2} holds true. 
\end{teo}
As an immediate consequence of the Selection Principle and of the
above theorem, we now prove \eqref{eq:HHWconj-2}.
\begin{teo}[Hall's conjecture]\label{teo:Hconj}
Let $E\in\Se^2$. Then \eqref{eq:HHWconj-2} holds true.
\end{teo}
\begin{proof}
By (iv) in Theorem \ref{teo:SP}, there exists a sequence of sets $(E_j)_j\subset\Se^2$ such that 
\begin{equation}\label{eq:ivSP}
Q(E_j)\to Q(B)\quad\text{and}\quad \|H_j-1\|_{L^{\infty}(\de E_j)}\to 0,
\end{equation} 
where $H_j$ stands for the curvature of $\de E_j$. This in particular
implies the existence of $j_0>0$, such that $E_j$ is a convex set for
all $j\geq j_0$ . By Theorem \ref{teo:HHW91} we have
\begin{eqnarray*}
Q(E_j)\geq C_0+o(1).
\end{eqnarray*}
Passing to the limit as $j$ tends to $\infty$, and thanks to \eqref{eq:ivSP}, we eventually get $Q(B)\geq C_0$ which in turn implies \eqref{eq:HHWconj-2} by the definition of $Q(B)$.
\end{proof}

\bibliographystyle{siam}

\end{document}